\documentclass[10pt]{amsart}
\usepackage[T1]{fontenc}
\usepackage[margin=1.2in,marginparsep=0.1in,marginparwidth=1in]{geometry}
\usepackage{amssymb,amsmath,amsthm,amstext,amscd,latexsym,graphics,graphicx,bbm,caption}
\usepackage[usenames,dvipsnames,svgnames,table]{xcolor}
\usepackage[plainpages=false,colorlinks=true]{hyperref}
\usepackage{tikz}
\usetikzlibrary{cd,positioning}

\tikzset{>=stealth}
\hypersetup{citecolor=Sepia,linkcolor=blue, urlcolor=blue}

\usepackage{cite}   

\makeatletter
\def\@tocline#1#2#3#4#5#6#7{\relax
	\ifnum #1>\c@tocdepth 
	\else
	\par \addpenalty\@secpenalty\addvspace{#2}%
	\begingroup \hyphenpenalty\@M
	\@ifempty{#4}{%
		\@tempdima\csname r@tocindent\number#1\endcsname\relax
	}{%
		\@tempdima#4\relax
	}%
	\parindent\z@ \leftskip#3\relax \advance\leftskip\@tempdima\relax
	\rightskip\@pnumwidth plus4em \parfillskip-\@pnumwidth
	#5\leavevmode\hskip-\@tempdima
	\ifcase #1
	\or\or \hskip 2em \or \hskip 2em \else \hskip 3em \fi%
	#6\nobreak\relax
	\dotfill\hbox to\@pnumwidth{\@tocpagenum{#7}}\par
	\nobreak
	\endgroup
	\fi}
\makeatother

\newtheorem{intro-thm}{Theorem}[]
\theoremstyle{plain}
\newtheorem{thm}{Theorem}[section]
\newtheorem{theorem}[thm]{Theorem}
\newtheorem{thmx}{Theorem}

\newtheorem{q}[thm]{Question}
\newtheorem{lemma}[thm]{Lemma}
\newtheorem{corollary}[thm]{Corollary}
\newtheorem{proposition}[thm]{Proposition}

\theoremstyle{definition}
\newtheorem{remark}[thm]{Remark}

\newtheorem{definition}[thm]{Definition}
\newtheorem{example}[thm]{Example}

\newcommand{\Proj}{{\P roj}}

\newcommand{\codim}{{\rm codim}}

\newcommand{\Spec}{{\rm Spec \,}}

\renewcommand{\tilde}{\widetilde}

\newcommand{\sE}{{\mathcal E}}

\newcommand{\sG}{{\mathcal G}}
\newcommand{\sH}{{\mathcal H}}
\newcommand{\sI}{{\mathcal I}}

\newcommand{\sK}{{\mathcal K}}

\newcommand{\sU}{{\mathcal U}}
\newcommand{\sV}{{\mathcal V}}

\newcommand{\sX}{{\mathcal X}}
\newcommand{\sY}{{\mathcal Y}}
\newcommand{\sZ}{{\mathcal Z}}

\newcommand{\G}{{\mathbb G}}

\renewcommand{\P}{{\mathbb P}}
\newcommand{\Q}{{\mathbb Q}}

\newcommand{\Z}{{\mathbb Z}}

\newcommand{\colim}{{\rm colim \,}}

\newcommand{\DM}[2]{\mathbf{DM}_{#2}^{\mathit{eff}}(#1)}

\newcommand{\DA}[2]{\mathbf{DA}_{#2}^{\mathit{eff}}(#1)}
\newcommand{\DAs}[2]{\mathbf{DA}_{#2}(#1)}
\newcommand{\twistedM}{M_\chi}
\newcommand{\SH}{\mathbf{SH}}

\usepackage{verbatim}

\input{xy}
\xyoption{all}

\begin{document}
\title{A Motivic Riemann-Roch Theorem for Deligne-Mumford Stacks}

\author{Utsav Choudhury}
\address{Stat Math unit, Indian Statistical Institute Kolkata, 203 Barrackpore Trunk Road, Kolkata 700108 India}
\email{prabrishik@gmail.com}

\author{Neeraj Deshmukh}
\address{Institute of Mathematics, Polish Academy of Sciences, Ul. \'{S}niadeckich 8, 00-656 Warsaw, Poland
}
\email{ndeshmukh@impan.pl}

\author{Amit Hogadi}
\address{Department of Mathematics, Indian Insititute of Science Education and Research (IISER) Pune, Dr. Homi Bhabha road, Pashan, Pune 411008 India}
\email{amit@iiserpune.ac.in}

\begin{abstract}
We develop a  motivic cohomology theory,  representable in the Voevodsky's triangulated category of motives,  for smooth separated Deligne-Mumford stacks and  show that the resulting higher Chow groups  are  canonically isomorphic to the higher $K$-theory of such stacks. This generalises the Grothendieck-Riemann-Roch theorem to the category of smooth Deligne-Mumford stacks. 
\end{abstract}

\maketitle

\section{Introduction}
Let $k$ be a field and let $X$ be a smooth $k$-scheme. The (higher) algebraic $K$-theory of $X$ is defined by the higher homotopy groups of the classifying space of a category associated with the category of vector bundles on $X$. On the other hand Bloch's (higher) Chow groups \cite{blochhigherchow}  are described as cohomology groups of Bloch's cycle complexes in terms of generators (special algebraic cycles on $X \times \mathbb{A}^n_k$) and relations (homotopies parametrized by the path object $\mathbb{A}^1$). The Grothendieck-Riemann-Roch theorem provides an isomorphism between the rational higher $K$-groups of $X$ and higher Chow groups \cite{blochhigherchow}. Later, Voevodsky proved that the higher Chow groups of $X$ are representable in the triangulated category of motives \cite{MWV},  $\DM{k, \Z}{}$. More precisely, there are motivic complexes $\Z(i)[n] \in  \DM{k, \Z}{}$ and $M(X) \in \DM{k, \Z}{}$ (the motive of $X$) such that we have the following chain of natural isomorphisms,
$$H^{n,i}(X, \Z) \cong Hom_{\DM{k, \Z}{}}(M(X), \Z(i)[n]) \cong CH^i(X, 2i-n)$$
between motivic cohomology groups (defined as hypercohomology of a complex of sheaves) on the left \cite{MWV} and higher Chow groups on the right. In particular, for smooth $k$-schemes we get a canonical isomorphism between suitable motivic cohomology groups and higher $K$-groups, rationally.

Algebraic stacks arise naturally in the study of families of algebraic varieties because many moduli problems do not give fine moduli space and therefore many important constructions (e.g, universal families) are not possible. Enlarging the category of schemes to include algebraic stacks resolve these issues. Therefore it is natural to search for extensions to the category of algebraic stacks of theorems which are true for varieties. In this article we extend the results mentioned in the previous paragraph to the category of smooth Deligne-Mumford stacks. 

Our first goal is to construct a suitable notion of higher Chow groups for a Deligne-Mumford stack rationally which is representable in $\DM{k, \Q}{}$. The integral Borel-style higher Chow groups of quotient stacks defined by Edidin and Graham is known to be representable in $\DM{k, \Z}{}$ (see \cite{cdh}). Rationally, these Borel-style higher Chow groups of a smooth Deligne-Mumford stack are isomorphic to the higher Chow groups of its coarse moduli space. Hence, it satisfies \'etale descent. Therefore, these higher Chow groups of Deligne-Mumford stacks are not suitable for Riemann-Roch type isomorphisms as $K$-theory of Deligne-Mumford stacks does not satisfy \'etale descent rationally. In \cite{toenthesis}, To\"{e}n proved that rationally $K$ theory of a smooth Deligne-Mumford stack $\mathcal{X}$ is isomorphic to the twisted (by character sheaf) \'etale $K$-theory of  $C^t_{\mathcal{X}}$, the stack classifying cyclic subgroups of automorphisms.  Inspired by this, we construct a motive $M_{\chi}(\mathcal{X}) \in \DM{k, \Q}{}$ of a smooth Deligne-Mumford stack  $\mathcal{X}$ and using this define motivic cohomology $H^{p,q}_{\chi}(\mathcal{X}, \Q)$. This is the main construction of this article. We use this construction to obtain the following motivic version of To\"{e}n's Grothendieck-Riemann-Roch isomorphism \cite{toengtheorynotes,Toen},

\begin{thmx}[see Theorem \ref{theorem-riemann-roch}]
	Let $F$ be a smooth separated Deligne-Mumford stack over a field $k$. Then the motivic Chern character induces an isomorphism 
	\[ch: K_i(F)_\Q\rightarrow \prod_j H_\chi^{2j-i,j}(F,\Q).\]
	The above isomorphism is covariantly functorial along the coarse moduli morphism $p: F\rightarrow X$.	
\end{thmx}

We also show that in case of smooth proper Deligne-Mumford stack $\mathcal{X}$, the Chow groups $H^{2*, *}_{\chi}(\mathcal{X}, \Q)$ is canonically isomorphic to To\"en's Chow-groups $CH^*_{\chi}(\mathcal{X},\Q)$ defined in \cite{Toen}.

The above theorem gives us an interpretation of the rational $K$-groups of $F$ in terms of the motivic cohomology of the motivic complexes $M_{\chi}(\sX)$. It is not immediately clear how the groups $H_\chi^{2j-i,j}(F,\Q)$ are related to the intersection theory of $F$. Nevertheless, one can show that after adding enough root of unity to the coefficient ring $\Q$, the groups $H_\chi^{2j-i,j}(F,\Q)$ become isomorphic to the rational motivic cohomology of the inertia stack of $\sX$.

\begin{thmx}[see Theorem \ref{comp inertia}]
	Let $\sX$ be a smooth separated Deligne-Mumford defined over an algebraically closed field. Denote by $I_\sX$ the inertia stack of $\sX$. Then we have an isomorphism,
	\[H_\chi^{2j-i,j}(\sX,\Q(\mu_\infty))\simeq H^{2j-i,i}(I_\sX,\Q(\mu_\infty)) \]
\end{thmx}
The above isomorphism is \textit{not} canonical and depends on the choice of roots of unity.

In the case when $\sX$ is a quotient stack of the form $[X/GL_n]$ with $X$ smooth quasi-projective, the above result tells us that the groups $H_\chi^{2j-i,j}(\sX,\Q(\mu_\infty))$ can be computed as the Edidin-Graham-Totaro higher Chow groups of the inertia stack $I_\sX$ with $\Q(\mu_\infty)$-coefficients (see \cite[Theorem 1.5]{cdh}). By \cite[Theorem 3.3]{MotivesDM}, when $\sX$ is tame, this is the same as computing the higher Chow groups of the coarse space $I_\sX$. This recover the result of Krishna-Sreedhar \cite{krishnasreedhar}. In fact, to the best of our knowledge, the above two Theorems recover all the currently known Riemann-Roch formalisms of  Krishna-Sreedhar \cite{krishnasreedhar}, Edidin-Graham \cite{edidingrahamriemannroch, edidingrahamnonabelian}, To\"{e}n \cite{toenthesis,Toen}, etc.\\

\noindent\textbf{Outline.} We begin by outlining the constructions of the basic motivic categories and functors in Section \ref{section-basic-constructions}. We primarily work with Ayoub's category of motivic sheaves without transfers $\DAs{\sX,\Q}{}$ and the rational motivic stable homotopy category $\textbf{SH}(\sX)_\Q$. Also, we freely use that equivalence between $\mathbf{DM}_{\acute{e}t}(-,\Q)$ and $\mathbf{DA}_{\acute{e}t}(-,\Q)$ as proved in \cite{ayoubetalerealisation,cd2019}.

In Section \ref{section-character-stack}, we present our main contruction. We use  stack $C^t_\sX$ (see Section \ref{section-character-stack}) to define the \textit{twisted motive} $M_\chi (\sX)\in \DM{k,\Q}{}$ associated to $\sX$.
 
In Section \ref{section-riemann-roch}, we prove the Grothendieck-Riemann-Roch isomorphism for smooth Deligne-Mumford stacks (Theorem \ref{theorem-riemann-roch}), where the target of the isomorphism is the new motivic cohomology $H^{p,q}_{\chi}$ defined in Section \ref{section-character-stack}. Further, in Theorem \ref{comp inertia}, we show that, after tensoring with $\Q(\mu_\infty)$, the new motivic cohomology  $H^{p,q}_{\chi}$ is precisely the motivic cohomology of the coarse moduli space of the inertia stack.

In Section \ref{section-motive-properties}, we prove some useful formulas for the new motive $M_{\chi}(\mathcal{X})$. While most of the usual properties of motives hold for \textit{stabiliser preserving morphisms}, the projective bundle formula fails to hold when the structure map $\P(\sE)\rightarrow \sX$ of the projective bundle is not stabiliser preserving (Example \ref{example-projective-bundle-formula-fails}). $M_{\chi}(\mathcal{X})$ is a geometric motive for smooth Deligne-Mumford stack $\mathcal{X}$ and it is a Chow motive if moreover $\mathcal{X}$ is proper. Using this, it is shown that $M_{\chi}(\mathcal{X})$ is isomorphic to $\iota \circ h_{\chi}(\mathcal{X})$ for $\mathcal{X}$ smooth and proper. Here $\iota$ is the inclusion of the category Chow motives into Voevodsky's category of motives and $h_{\chi}(\mathcal{X})$ is the Chow motive of $\mathcal{X}$ defined by To\"en (see \cite{Toen}).

Finally, in Section \ref{section-computations}, we discuss some explicit examples of the motive $M_\chi (\sX)$.\\

\noindent\textbf{Acknowledgements.} The second author was supported by the project KAPIBARA funded by the European Research Council (ERC) under the European Union's Horizon 2020 research and innovation programme (grant agreement No 802787). He also benefitted from the support of the Swiss National Science Foundation (SNF), project 200020\_178729. A part of this work was carried out in Indian Statistical Institute, Kolkata, when the second author was visiting the first author and they thank the institute for excellent academic environment.

The second author thanks Joseph Ayoub and Andrew Kresch for useful conversations around the subject of this work.

\section{Basic constructions} \label{section-basic-constructions}
Let $k$ be a field. For any  Deligne-Mumford stack $\sX$, let $Sm/\sX$ denote the category of smooth schemes over $\sX$ and let $Smrep/\sX$ be the category of smooth Deligne-Mumford stack over $\sX$ such that the structre map is representable.

Let $M$ be either the category of simplicial sets equipped with the usual model structure or the category of unbounded chain complexes of $R$-modules ($R = \mathbb{Z}, \mathbb{Q}$), equipped with the projective model structre.  For any category $C$ , let $\textrm{PSh}(C, M)$ denote the category of presheaves on $C$ taking values in $M$.  Let $e : \textrm{PSh}(Sm/\sX, M) \to \textrm{PSh}(Smrep/\sX, M)$ be the extension of presheaf functor which is left adjoint to the restriction $r : \textrm{PSh}(Smrep/\sX, M) \to \textrm{PSh}(Sm/\sX, M)$ induced by the inclusion $Sm/\sX \to  Smrep/\sX$. Note that this gives a Quillen adjunction , where both sides are equipped with the global projective model structure induced from $M$.

Let $\DAs{\sX,\Q}{\acute{e}t}$ (resp. $\DA{\sX,\Q}{\acute{e}t}$) be the category of stable etale motives without transfers with $\mathbb{Q}$ coefficients ( resp. effective etale motives without transfers with $\mathbb{Q}$-coefficients) with respect to the etale stable $\mathbb{A}^1$-(projective ) model structure on  the category of $T$-spectra  on $\textrm{PSh}(Sm/\sX, M)$ (resp. $\mathbb{A}^1$-\'etale model structure on $ \textrm{PSh}(Sm/\sX, M)$), where $M$ is the category of chain complexes of $\mathbb{Q}$-vector spaces. Same construction could be done on $T$-spectra on   $\textrm{PSh}(Smrep/\sX, M)$ (resp. on $ \textrm{PSh}(Smrep/\sX, M)$) to get the category of stable etale motives without transfer with $\mathbb{Q}$-coefficients $\DAs{Smrep/\sX, \Q}{\acute{e}t}$ (resp. effective etale motives without transfers with $\mathbb{Q}$-coefficients 
$\DA{Smrep/\sX, \Q}{\acute{e}t}$). 

\begin{lemma} \label{extension}
The adjunction $(e,r)$ induces a Quillen equivalence $$\DAs{\sX, \Q}{\acute{e}t} \cong \DAs{Smrep/\sX, \Q}{\acute{e}t}$$ (resp. $\DA{\sX, \Q}{\acute{e}t} \cong \DA{Smrep/\sX, \Q}{\acute{e}t}$).
\end{lemma}

\begin{proof}
We will show that the Quillen adjunction $$e\!:\!\textrm{PSh}(Sm/\sX, M) \rightleftarrows \textrm{PSh}(Smrep/\sX, M)\!:\!r$$  induces Quillen equivalence $$e:\DA{\sX, \Q}{\acute{e}t} \rightleftarrows \DA{Smrep/\sX, \Q}{\acute{e}t}:r.$$ We will show that $Le \circ Rr \cong Id$ and $Rr \circ Le \cong Id$. Here $Le$ is $\mathbb{A}^1$ left derivation of $e$ and $Rr$ is $\mathbb{A}^1$-right derivation of $r$. Note that for any presheaf $F$ on $Sm/\sX$, $F \to r \circ e (F)$ is an isomorphism. Let $K$ be a $\mathbb{A}^1$-fibrant cofibrant presheaf of complex of vector spaces on $Sm/\sX$. Let $e(K) \to K'$ be an \'etale fibrant replacement. Then $K \cong re(K) \to r(K')$ is an \'etale weak equivalence. As $r(K')$ is also \'etale fibrant, the morphism $K \to r(K')$ is an \'etale weak equiavelnce of \'etale fibrant objects. This implies that $K \to r(K')$ is sectionwise weak equivalence. Therefore to show $Id \cong Rr \circ Le$, it is enough to show that $K'$ is $\mathbb{A}^1$-fibrant. As $K'$ is \'etale fibrant, we need to show that for any $\sY \to \sX$ , smooth representable, the canonical morphism $\underline{Hom}(\sY, K') \to \underline{Hom}(\mathbb{A}^1_{\sY}, K')$ is an \'etale local equivalence. But $rK'$ is $\mathbb{A}^1$-fibrant being sectionwise equivalent to $K$. On the other hand for any smooth simplicial scheme $U_{\bullet}$ over $\sX$, we have
$$\underline{Hom}(U_{\bullet}, rK') \cong \underline{Hom}(U_{\bullet}, K').$$ Let $U_{\bullet} \to \sY$ be an \'etale weak equivalence where $U_{\bullet}$ is a smooth simplicial scheme over $\sX$. Then we have the following commutative diagram

\begin{center}
		\begin{tikzcd}
		\underline{Hom}(\sY, K') \arrow[r," p'"]\arrow[d, "\tilde{f}"] & \underline{Hom}(\mathbb{A}^1_{\sY}, K')  \arrow[d, "f"]\\
		 \underline{Hom}(U_{\bullet}, K') \arrow[r, "p"] &  \underline{Hom}(\mathbb{A}^1 \times U_{\bullet}, K')
		\end{tikzcd}
	\end{center}

such that the vertical arrows are etale local weak equivalence  by \'etale descent and the bottom horizontal arrow is an \'etale local weak equivalence. Therefore the top horizontal morphism is a local weak equivalence. Thus we get $Id \cong Rr \circ Le$.

Now we want to show $Le \circ Rr \cong Id$. This follows from the fact that for any  presheaf $S$ on $Smrep/\sX$, $e \circ r (S) \to S$ induces isomorphism on \'etale stalks.

\end{proof}

\begin{remark}
Let $M$ be the category of simplicial sets with the usual model structure. The unstable $\mathbb{A}^1$-etale homotopy category $\mathbf{H}_{et}(\sX)$ (resp.$\mathbf{H}_{et}(Smrep/\sX)$)  be the homotopy category with respect to the $\mathbb{A}^1$- \'etale model structure on $\textrm{PSh}(Sm/\sX, M)$ (resp. on $\textrm{PSh}(Smrep/\sX, M)$). Then the adjunction $(e,r)$ induces Quillen equivalence
$$\mathbf{H}_{et}(\sX) \cong \mathbf{H}_{et}(Smrep/\sX).$$

\end{remark}

\begin{remark}

Let $f : \sX \to \sY$ be a morphism of  DM stacks. This induces adjunction $$(f^*, f_*) :  \textrm{PSh}(Smrep/\sY, M) \rightleftarrows  \textrm{PSh}(Smrep/\sX, M),$$ which is a Quillen adjunction with repect to the $\mathbb{A}^1$-model structure on source and the target. If $M$ is the category of simplicial set, then any $\sY' \to \sY$ smooth representable, we have $f^*(\sY') = \sY' \times_{\sY} \sX$. This gives a pair of adjoint functors $$(Lf^*, Rf_*) :  \DA{Smrep/\sY, \Q}{\acute{e}t} \rightleftarrows  \DA{Smrep/\sX, \Q}{\acute{e}t}$$  $$\text{(resp. } (Lf^*, Rf_*) :  \DAs{Smrep/\sY, \Q}{\acute{e}t}  \rightleftarrows \DAs{Smrep/\sX, \Q}{\acute{e}t}.)$$

Composing with the equivalences mentioned in the lemma \ref{extension} we get adjunctions, which we denote by $$(Lf^*, Rf_*) : \DA{\sY, \Q}{\acute{e}t} \leftrightarrows \DA{\sX, \Q}{\acute{e}t},$$ $$\text{ (resp. } (Lf^*, Rf_* ):  \DAs{\sY, \Q}{\acute{e}t} \leftrightarrows \DAs{\sX, \Q}{\acute{e}t}.)$$

If $f : \sX \to \sY$ happens to be smooth then we get a Quillen adjunction 
$$(f_{\sharp}, f^*) :  \textrm{PSh}(Sm/\sX, M) \leftrightarrows  \textrm{PSh}(Sm/\sY, M),$$ 
such that 
$$f_{\sharp}(U \to \sX) = U \to \sX \to \sY.$$ 
This gives an adjunction $$Lf_{\sharp}, Rf^* : \DA{\sX, \Q}{\acute{e}t} \leftrightarrows \DA{\sY, \Q}{\acute{e}t}$$  $$\text{( resp. }(Lf_{\sharp}, Rf^*) : \DAs{\sX, \Q}{\acute{e}t} \leftrightarrows \DAs{\sY, \Q}{\acute{e}t}.)$$

\end{remark}

\begin{remark}

Let $\sK$ be the $K$-theory spectrum and $k$ be a perfect field. By \cite[6.2.3.6]{riouthesis} and the fact that $H^{2j-1}(k, \Q(j)) = 0$ for all $j$ and $H^{2j}(k, \Q(j)) =0$ for $j \neq 0$ we get $Hom_{SH(k)}(\sK, H_{\Q}) \cong \Q$. Let $ch$ be the canonical generator, therefore it can be lifted to give a map $ch : \sK \to H_{\Q}$, using Bott periodicity and adjunction we get
sequence of maps $ch_i : \sK_{\Q} \to  \mathbf{1}(i)[2i]_\Q$ of spectrums. 
The total chern character induces an isomorphism (see \cite[Remark 6.2.3.10]{riouthesis})
	\[ch_t: \sK_\Q \simeq \prod_j \mathbf{1}(j)[2j]_\Q.\]

Note that the infinite product is also a direct sum.

If $\sX$ is a stack over $k$, we have a pull back functor along the structure morphism $f: \sX\rightarrow \Spec k$,
\[f^*:\SH(k)\rightarrow \SH(\sX)\]

This allows us to pull back the above isomorphism $ch_t$ over $\sX$.\\

\end{remark}

\section{The stack $C^t_F$}\label{section-character-stack}

In this section, we describe the stack $C^t_F$ classifying cyclic subgroups of automorphism associated to any Deligne-Mumford stack $F$ as defined in \cite{Toen}. We call it the character stack of $F$.

Let $F$ be a tame Deligne-Mumford stack over a field $k$. The stack $C^t_F$ is defined as follows:

For any $k$-scheme $U$, the objects of $C^t_F$ are pairs $(u,C)$, where $u: U\rightarrow F$ and $C\subset Aut(u)$ is a cyclic subgroup. An isomorphism between any two such objects $(u,C)$ and $(u',C')$ is described by an isomorphism $\alpha: u\rightarrow u'$ which conjugates the cyclic subgroups, i.e, $\alpha^{-1}C'\alpha = C$. We have a canonical map $\pi: C^t_F \rightarrow F$ which sends any $(u,C)$ to $u$.

We can define the sheaf of characters $\chi$ on the big \'{e}tale site of $C^t_F$ as
\begin{center}
$\begin{matrix}
\chi: & (Sch/C^t_F)_{\acute{e}t} &\rightarrow & \text{Ab}\\
& ((u,C):U\rightarrow C^t_F) &\mapsto  &Hom_{Gps}(C, \G_{m,U})
\end{matrix}$
\end{center}

Consider the sheaf of group-algebras $\Q[\chi]$  (resp. $\Z[\chi]$) associated to $\chi$. This is \'etale locally (on $(Sch/C^t_F)_{\acute{e}t}$)  isomorphic to the constant sheaf with fibre $\Q[T]/(T^m-1)$ ($\Z[T]/(T^m-1)$). We have maps $p:\Q[T]/(T^m-1)\rightarrow \Q(\zeta_m)$ (p :  $\Z[T]/(T^m -1) \to \Z[T]/(\psi_m)$ where $\psi_m$ is the minimal polynomial of the primitive $m$-th root of unity), where $\Q(\zeta_m)$ is a cyclotomic extension of $\Q$. The kernels $ker(p)$ glue to give a well-defined ideal sheaf, $\sI$. We define a sheaf 
\[\Lambda_F:= \Q[\chi]/\sI \;\;\;\;(resp.\,\, \Lambda_{F, \Z}:= \Z[\chi]/\sI)\]

The above discussion culminates into the following:

\begin{lemma}
	$\Lambda_F$  (resp. $\Lambda_{F, \Z}$) is a locally constant sheaf on $ (Sch/C^t_F)_{\acute{e}t}$ and it is locally given by the constant sheaf $\Q(\zeta_m)$ (resp. $\Z[\zeta_m]$).
\end{lemma}

Let $\DAs{C^t_F,\Q}{\acute{e}t}$ (respectively $\DAs{C^t_F,\Z}{\acute{e}t}$) be the triangulated category of motives without transfers in the \'{e}tale topology (see \cite{ayoub-hopf}). 

\begin{lemma}
$\Lambda_F$ is a dualisable object in $\DAs{C^t_F,\Q}{\acute{e}t}$ (Also true for $\Lambda_{F, \Z}$).
\end{lemma}

\begin{proof}
	By the previous lemma, $\Lambda_F$ is locally constant. Thus, there is an \'{e}tale covering $p: U\rightarrow C^t_F$ such $\Lambda_F$ is free of finite rank. Then by \'{e}tale descent, we get the required result.
\end{proof}

\begin{definition} (New higher Chow groups and motivic cohomology) We define the $\Lambda$-higher Chow groups as
\[CH^j(C^t_F, i)_{\Lambda_F}:= Hom_{C^t_F}(\mathbf{1}, \Lambda_F(j)[2j-i])\]
\end{definition}

\begin{proposition}
The $\Lambda$-higher Chow groups defined above are representable in $\DAs{k,\Q}{\acute{e}t}$. That is, there exists a motive $N\in \DAs{k,\Q}{\acute{e}t}$ such that $Hom_{\DAs{k,\Q}{\acute{e}t}}(N, \Q(j)[2j-i])\simeq CH^j(C^t_F,i)_{\Lambda_F}$.
\end{proposition}
\begin{proof}
	As $\Lambda_F$ is a dualisable object in $\DAs{C^t_F,\Q}{\acute{e}t}$, we have that
	\begin{align*}
	Hom_{C^t_F}(\mathbf{1}_{C^t_F}, \Lambda_F(j)[2j-i]) &= Hom_{C^t_F}(\mathbf{1}_{C^t_F}\otimes \Lambda_F^\vee, \mathbf{1}(j)[2j-i])\\
	&= Hom_{k}(f_\#(\mathbf{1}_{C^t_F}\otimes \Lambda_F^\vee), \mathbf{1}(j)[2j-i])
	\end{align*}
	where $f: C^t_F\rightarrow \Spec k$ is the structure map. Thus, the $\Lambda$-motivic cohomology of $C^t_F$ is represented by the motive $f_\#(\mathbf{1}_{C^t_F}\otimes \Lambda_F^\vee)$ in $\DAs{k,\Q}{\acute{e}t}$
\end{proof}

\begin{remark}
	The above proposition is also true with $\Z$ coefficients in $\DAs{k,\Z}{\acute{e}t}$ and we get an integral \'{e}tale motive.
\end{remark}

Let $F$ be a Deligne-Mumford stack. The motive $M(F)$ as described in \cite{MotivesDM, cdh} is ``too small" for computing its $K$-Theory. To get a reasonable comparison with $K$-Theory, we need to ``enlarge" the motive of $F$. 
We make the following definition.

\begin{definition}
	Let $F$ be a Deligne-Mumford stack and consider the associated Character stack $f: C^t_F\rightarrow \Spec k$. We define the motive 
	\[M_\chi(F):= M( f_\# (\mathbf{1}_{C^t_F}\otimes \Lambda_F^\vee))\]
	We denote the associated motivic cohomology by
	\[H_\chi^{p,q}(F, \Q):=Hom\big(M_\chi(F), \Q(q)[p]\big)\]
\end{definition}

\begin{proposition}[Basic functoriality]
\begin{enumerate}
\item Let $f : F \to F'$ be a morphism of smooth Deligne-Mumford stacks oevr $k$. Then $f$ induces a natural morphism  $M_{\chi}(f) : M_\chi(F) \to M_\chi(F')$ such that we get a functor $M_{\chi} : SmDM/k \to 
\DM{k,\Q}{}$ sending $F \mapsto  M_\chi(F)$ and $f : F \to F' \mapsto M_{\chi}(f)$. 

\item Let $ f : Spec(L) \to Spec(k)$ be a separable field extension and let $F$ be smooth DM stack over $k$. Then $f^*(M_\chi(F)) \cong M_\chi(F \times_k L)$.
\end{enumerate}
\end{proposition}

\begin{proof}

(1)
Let $Cf : C^t_F \to C^t_{F'}$ be the induced map. This map is given in the following way :
 for any $k$-scheme $X$, $x \in F(X)$ and let $c \subset Aut(x)$ be a cyclic subgroup scheme over $X$, $Cf(x,c) = (f(x), f(c))$. Consider the inverse image  functor $Cf^* : PSh(Sm/C^t_{F'}) \to PSh(Sm/C^t_{F})$. For any persheaf $G \in  PSh(Sm/C^t_{F'})$ and $U \in Sm/C^t_{F}$ we have 
$$Cf^* (G)(U) = colim_{U \to V \times_{C^t_{F'}} C^t_F \in Smrep/{C^t_F}, V \in Sm/{C^t_{F'}}} G(V).$$  Note that for any $V \in Sm/C^t_{F'}$, the presheaf $Cf^*(V)$ is the sheaf represented by the object  $V \times_{C^t_{F'}} C^t_F \to C^t_F$. Restriction of characteres thus gives a morphism $res_f : Cf^*(\chi_{F'}) \to \chi_{F}$, which induces a morphism of presheaf of algebras $res_f : Cf^*(\Lambda_{F'}) \to \Lambda_F$. As $Cf^*$ is monoidal we get a morphism
$$t : Cf^*(\Lambda_{F'}^\vee) \to Cf^*(\Lambda_{F'})^\vee.$$ This morphism is invertible as $\Lambda_{F'}$ is strongly dualisable. Thus we get a morphism $\theta := t^{-1} \circ res_f^\vee : \Lambda_F^\vee \to Cf^*(\Lambda_{F'}^\vee)$. Let $p : C^t_F \to Spec(k)$ (resp. $p' : C^t_{F'} \to Spec(k)$) denote the structure map. Then $p'^*$ has a left adjoint $p'_{\sharp}$. Thus we get a canonical map $\Lambda_{F'}^\vee \to p'^* p'_{\sharp}(\Lambda_{F'}^\vee)$. Composing with with $\theta$ we get
$$ \Lambda_F^\vee \to Cf^*  p'^* p'_{\sharp}(\Lambda_{F'}^\vee) \cong p^* p'_{\sharp}(\Lambda_{F'}^\vee).$$ By adjunction, this gives us
$$p_{\sharp}(\Lambda_F^\vee) \to p'_{\sharp}(\Lambda_{F'}^\vee).$$ This is the required morphism $M_{\chi}(f) : M_\chi(F) \to M_\chi(F')$. \\

(2)
 As $f$ is smooth, $f^*(A)$ is just restriction of any presheaf $A$ to $Sm/L$. Therefore, it is easy to verify $$C^t_{f^*F} \cong  C^t_{F \times_k L} \cong C^t_F \times_k L \cong f^*(C^t_F)$$  as stacks on $Sm/L$.  Let $\tilde{f} : C^t_F \times_k L \to C^t_F$ be the canonical morphism which gives us the following cartesian diagram of stacks
\begin{center}
		\begin{tikzcd}
		 C^t_{F \times_k L}\arrow[r," p'"]\arrow[d, "\tilde{f}"] & Spec(L)\arrow[d, "f"]\\
		C^t_F \arrow[r, "p"] & Spec(k)
		\end{tikzcd}
	\end{center}

Note that $\tilde{f}$ is smooth and $$\tilde{f}^*(\Lambda_F) \cong \Lambda_{f^*F}$$ and therefore $$\tilde{f}^*(\Lambda_{F}^\vee) \cong \Lambda_{f^*F}^\vee.$$ The counit  $\tilde{f}_{\sharp} \tilde{f}^* \to Id$ of the adjunction $(\tilde{f}_{\sharp},  \tilde{f}^*) $ induces the following morphism $f_{\sharp} p'_{\sharp} \tilde{f}^* \cong p_{\sharp}\tilde{f}_{\sharp} \tilde{f}^* \to p_{\sharp}$. This by adjunction $(f_{\sharp}, f^*)$, induces a morphism  $  p'_{\sharp} \tilde{f}^* \to f^* p_{\sharp}$. Both the functors preserves colimit and evaluating on $U \in Sm/ C^t_F$, we get $  p'_{\sharp} \tilde{f}^* \to f^* p_{\sharp}$ is invertible. Therefore, $$f^* p_{\sharp}(\Lambda_{F}^\vee) \cong   p'_{\sharp} \tilde{f}^*(\Lambda_{F}^\vee) \cong  p'_{\sharp}(\Lambda_{f^*F}^\vee).$$ This is our required isomorphism. 
\end{proof}

\begin{proposition}
	There exists a decomposition $M_\chi(F)= M(F)\oplus M_{\chi \neq 1}(F)$. 
	This induces the projection map $H_\chi^{p,q}(F,\Q)\rightarrow H^{p,q}(F,\Q)$.
\end{proposition}
\begin{proof}
	The sturcture map $C^t_F\rightarrow F$ admits a section $e: F\rightarrow C^t_F$ given by the identity element. $e(F)$ defines a connected component $C^t_F$. Moreover, $e^*(\Lambda_F)= \Q$. Thus, we have a decomposition $\mathbf{1}_{C^t_F}\otimes \Lambda_F^\vee= (e_{\sharp}(\mathbf{1}_F)\otimes \Q)\oplus \mathcal{G}$. Applying $f_\#$, gives us the required result.
\end{proof}

The following theorem show that the new motive differs from the motive of inertia stack upto ``adding all roots of unity". Thus, the previous theorem recovers existing Riemann-Roch formalisms in Krishna-Sreedhar \cite{krishnasreedhar}, Edidin-Graham \cite{edidingrahamriemannroch}, Toen \cite{toenthesis,Toen}, etc.

\begin{theorem}\label{comp inertia}
	Let $k=\bar{k}$.
	\[M_\chi(F)\otimes \Q(\mu_\infty)\simeq M(I_F)\otimes \Q(\mu_\infty)\]
\end{theorem}

\begin{proof}
	For any Deligne-Mumford stack $F$, there is canonical map $p:I_F\rightarrow C^t_F$ given by sending any automorphism to the subgroup generated by it. This map is finite \'{e}tale and surjective. We claim that \[p_*p^*(\Q(\mu_\infty))\simeq \Lambda_F \otimes \Q(\mu_\infty).\]
	Then the result will follow by taking duals.
	
	Note that $p^*\Lambda_F$ is a sheaf of $\Q$-algebras on $I_F$.
	Choose an embedding of $\Lambda_F\hookrightarrow \Q(\mu_\infty)$ by choosing primitive roots of units.
	There is a map $p^*(\Lambda_F \otimes \Q(\mu_\infty))\rightarrow p^*\Q(\mu_\infty)=\Q(\mu_\infty)$. By adjunction, this gives us a map, $\Lambda_F \otimes \Q(\mu_\infty))\rightarrow p_*p^*\Q(\mu_\infty)$.
	As $F$ is a DM stack it admits an \'{e}tale cover $F'\rightarrow F$ such that $F'$ is a quotient of a finite group $H$ acting on a scheme $X$ (see \cite[Lemma 2.2.3]{stablemaps}. Thus, we are reduced to case when $F=[X/H]$. Consider the map $\tilde{p}:\sqcup_{x\in H} X\rightarrow \sqcup_{\sigma \subset H}X$ where $\sigma$ is a cyclic subgroup of $X$. This is an $H$-equivariant map whose quotient is $p:I_F\rightarrow C^t_F$. Then $p_*(\Q(\mu_\infty))$ on any component is a disjoint union $\sqcup_{<x>=\sigma}\Q(\mu_\infty)$ over all element of $H$ that generate $\sigma$. Using the chosen embedding $\Lambda\hookrightarrow \Q(\mu_\infty)$, we get the required result.
\end{proof}

\begin{remark}
In fact, for a choice of a particular Deligne-Mumford stack $F$, we can replace $\Q(\mu_\infty)$ with $\Q(\mu_n)$ for some large $n$, in the above theorem. The choice of $n$ depends on the stack $F$.
\end{remark}


	




\begin{theorem} \label{geometric}
For any smooth separated Deligne-Mumford stack $F$ over a field of characteristic $0$ the motive $M_\chi(F)$ is a geometric motive. If $F$ is moreover proper then $M_\chi(F)$ is a Chow motive.
\end{theorem}

\begin{proof}



If $c: \Spec L \rightarrow \Spec k$ is a Galois extension, $c^*(M_\chi (F))\simeq M_\chi (F_L)$ and $c_\# c^*(M_\chi (F))$ contains $M_\chi (F)$ as a factor, by Galois invariance.

There exists a Galois extension $c: \Spec L\rightarrow \Spec k$ and an $n>0$ such that there is a split injection $M_chi (F_L)\hookrightarrow M_\chi (F_L)\otimes \Q(\mu_n)\simeq M(I_F)\otimes \Q (\mu_n)$.

Now, $M(I_F)$ is a factor of $M(X)$ for some variety $X$ (if $F$ is proper, so is $X$). Thus, $M_\chi (F)$ is the factor of a geometric motive (chow motive if $F$ is proper). This implies that $c_\# M_\chi (F_L)$ is a factor of a geoemtric motive (respectively, Chow motive). Finally, as $M_\chi (F)$ is a factor of $c_\# M_\chi (F_L)$, $M_\chi (F)$ is also a factor of a geometric motive (respectively, Chow motive).
\end{proof}

\begin{theorem}
	Let $F$ be a smooth proper Deligne-Mumford stack over a field $k$ of characteristic zero. Then $M_\chi (F) \simeq i (h_\chi (F))$, where $h_\chi (F)$ is To\"{e}n's Chow motive and $i$ is the composition of the isomorphisms, \[i:\mathbf{ChowMot}^{To\ddot{e}n}_\chi (k)\rightarrow\mathbf{ChowMot}(k) \rightarrow \mathbf{DM}^{gm}(k)\]
\end{theorem}
\begin{proof}
	By the previous theorem, $M_\chi (F)$ is a Chow motive. By Manin's identity principle \cite[Theorem A.1]{MotivesDM}, it is enough to construct a natural isomorphism between the motivic cohomology groups $H_\chi^{p,q}(F,\Q)$ and To\"{e}n's Chow groups. Such an isomorphism exists by construction of To\"{e}n's Chow groups.
\end{proof}


\section{Riemann Roch}\label{section-riemann-roch}
We will now describe a motivic version of the Riemann-Roch theorem for Deligne-Mumford stacks.

\begin{theorem}\label{theorem-chow}
	Let $F$ be a smooth separated Deligne-Mumford stack over $k$. There is a canonical isomorphism
	\[K_i(F)_\Q \simeq \prod_j CH^j (C^t_F, i)_{\Lambda_F} \]
\end{theorem}

\begin{proof}
	The total chern character induces an isomorphism (see \cite[Remark 6.2.3.10]{riouthesis})
	\[ch: BGL_\Q\simeq \prod_j \mathbf{1}(j)[2j]_\Q.\] Tensoring with the sheaf $\Lambda_F$ of $\Q$-algebras and taking cohomology on $C^t_F$, we get
	\[H^{-i}_{\acute{e}t}(C^t_F, (\prod_j \mathbf{1}(j)[2j])_{\Lambda_F}) \simeq H^{-i}_{\acute{e}t}(C^t_F, BGL_{\Lambda_F}) \]
	By \cite[Theorem 2.4]{toengtheorynotes}, the right hand side is isomorphic to $K_i(F)_\Q$. Furthermore, the left hand side is 
	\begin{align*}
	H^{-i}_{\acute{e}t}(C^t_F, (\prod_j \mathbf{1}(j)[2j])_{\Lambda_F}) &\simeq Hom_{\DAs{C^t_F,\Q}{\acute{e}t}}(\Lambda_F^{\vee}[i], \prod_j \mathbf{1}(j)[2j] )\\
    &\simeq Hom_{\DAs{k,\Q}{\acute{e}t}}( M_{\chi}(F)[i],  \prod_j \mathbf{1}(j)[2j])\\
	&\simeq  \prod_j  Hom_{\DAs{k,\Q}{\acute{e}t}}( M_{\chi}(F)[i], \mathbf{1}(j)[2j])\\
	&(\simeq \prod_j CH^j(C^t_F, i)_{\Lambda_F})
	\end{align*}
\end{proof}

\begin{theorem}\label{theorem-riemann-roch}
	Let $F$ be a smooth separated Deligne-Mumford stack over a field $k$. Then the motivic Chern character induces an isomorphism 
	\[ch: K_i(F)_\Q\rightarrow \prod_j H_\chi^{2j-i,j}(F,\Q).\]
	Moreover, if $p:F\rightarrow X$ is the coarse moduli space, then we have a commutative diagram
	\begin{center}
		\begin{tikzcd}
		K_i(F)_\Q\arrow[r]\arrow[d] & \prod_j H_\chi^{2j-i,j}(F,\Q)\arrow[d]\\
		K_i(X)_\Q \arrow[r] & \prod_j H^{2j-i,j}(X,\Q)
		\end{tikzcd}
	\end{center}
	Here vertical arrows are the projection maps induced by the identity section $e: F\rightarrow C^t_F$.
\end{theorem}

\begin{proof}
	By the previous theorem, $K_i(F)_\Q\simeq \prod_j H^{2j-i,j}(C^t_F, \Lambda_F)$.
	Unravelling the right-hand side will give us the desired result.
	\begin{align*}
	H^{2j-i,j}(C^t_F, \Lambda_F) &= Hom_{\SH_\Q(C^t_F)}\big(\mathbf{1}_{C^t_F},\Lambda_F(j)[2j-i])\big)\\
	&= Hom_{\SH_\Q(C^t_F)}\big(f^*\mathbf{1}_k,\Lambda_F\otimes f^*\mathbf{1}_k(j)[2j-i])\big)\\
	&= Hom_{\SH_\Q(C^t_F)}\big(f^*\mathbf{1}_k \otimes \Lambda_F^\vee, f^*\mathbf{1}_k(j)[2j-i])\big)\\
	&= Hom_{\SH_\Q(k)}\big(f_\#(f^*\mathbf{1}_k \otimes \Lambda_F^\vee), \mathbf{1}_k(j)[2j-i])\big)
	\end{align*}
	Finally, to get the commutative diagram, we will show that the bottom vertical arrow is a summand of the top vertical arrow.
	
	Firstly, observe that we have canonical isomorphisms $K_{\acute{e}t}(F)_\Q\simeq K_i(X)_\Q$ and $M(F)_\Q\simeq M(X)_\Q$, that respect the Chern character. Thus, it suffices to show that $K_{\acute{e}t}(F)_\Q$ and $\mathbf{1}_F$ are summands of $K_{\acute{e}t}(C^t_F)_{\Lambda_F}$ and $\mathbf{1}_{C^t_F}\otimes \Lambda_F$ respectively in $\SH(C^t_F)$.
	To see this, note that the map $p: C^t_F \rightarrow F$ admits a section and we have a decomposition $C^t_F = F\sqcup C^t_{F,\chi \neq 1} $. This induces a splitting on $\SH(C^t_F)$ and, therefore, on the Chern character, $ch_{C^t_F}= ch_F \oplus ch_{\chi \neq 1}$. Thus, we have a commutative diagram, 
	\begin{center}
		\begin{tikzcd}
			BGL_{C^t_F}\otimes \Lambda_F\arrow[r,"ch"]\arrow[d]&\prod_j \mathbf{1}_{C^t_F}(j)[2j] \otimes {\Lambda_F}\arrow[d]\\
			BGL_F \arrow[r,"ch"] & \prod_j \mathbf{1}_F(j)[2j]
		\end{tikzcd}
	\end{center}
	Finally, taking cohomology on $C^t_F$, and precomposing the top horizontal arrow with the isomorphism $K_i(F)_\Q\simeq H^{-i}_{\acute{e}t}(C^t_F, BGL_{\Lambda_F})$ of \cite[Theorem 2.4]{toengtheorynotes}, we have the result.
\end{proof}

\begin{remark}
	As noted in \cite{riouthesis}, the infinite product $\prod_j \mathbf{1}_F(j)[2j]$ is acutally also a direct sum $\oplus_j \mathbf{1}_F(j)[2j]$. And in characteristic zero, the above results also hold for $\oplus \mathbf{1}_F(j)[2j]$.
\end{remark}

\begin{corollary}
	Let $k$ be a field of characteristic zero. Then Theorem \ref{theorem-chow} and \ref{theorem-riemann-roch} hold with the infinite product replaced with the infinite direct sum.
\end{corollary}

\begin{proof}
	We know that $\prod_j \mathbf{1}_F(j)[2j]\simeq \oplus_j \mathbf{1}_F(j)[2j]$. Thus, the only only thing left to prove is that  the infinite direct sum $\oplus_j \mathbf{1}_F(j)[2j]$ commutes with $Hom_{\DA{k,\Q}{\acute{e}t}}(M_\chi (F),-)$. But by Theorem \ref{geometric}, $M_\chi (F)$ is compact so the result follows.
\end{proof}

\section{Properties of the motive}\label{section-motive-properties}

While the new motive which we have defined gives a comparison with $K$-theory, there are other interesting things which happen. For instance it fails to have a projective bundle, and many of the exact triangles only work for stabiliser preserving morphisms.

\begin{example}[Failure of the Projective bundle formula]\label{example-projective-bundle-formula-fails}
	The motive constructed using the character stack does not have a projective bundle formula. Consider the stack $BC_2$ in characteristic different from $2$. Consider the projective bundle on $BC_2$ defined by action $C_2$ on $\P^1$, given by $[x:y]\mapsto [x:-y]$. The induced map on the character stacks $C^t_{[\P^1/C_2]}\rightarrow C^t_{BC_2}$ has the form:
	\begin{center}
	\begin{tikzcd}[column sep =.01ex]
	 {[\P^1/C_2]}	& \coprod\arrow[d] & BC_2 \sqcup BC_2\\
	 BC_2 & \coprod & BC_2
	\end{tikzcd}
	\end{center}
	where $[\P^1/C_2]$ is trivial component while the non-trivial component is given by two copies of $BC_2$. Now, $M_\chi (BC_2):= M(BC_2)\oplus M(BC_2)$. On the other hand, $M_\chi([\P^1/C_2])= M(BC_2)\oplus M(BC_2)(1)[2] \oplus M(BC_2) \oplus M(BC_2) $.
	
	Note, however, that it still produces the correct ranks in cohomology. Thus, the projective bundle formula continues to hold for $K$-groups. 
\end{example}

The failure of projective bundle formula is a direct consequence of the fact that a general representable morphism between stacks only induces an injection on automorphism groups. For stabilizer perserving morphisms, the new motive behaves exactly like the ordinary motive.

\begin{definition}
	A morphism of stacks $\sX\rightarrow \sY$ is said to be stabiliser preserving if the induced map on the inertia stacks $I_\sX \rightarrow I_\sY\times_\sY \sX$ is an isomorphism. Note that such a map automatically induces an isomorphism $C^t_\sX\rightarrow C^t_\sY\times_\sY \sX$.
\end{definition}

\subsection{Properties of the motive}

The twisted motive satisfies descent for stabiliser preserving Nisnevich distinguished squares, i.e, Nisnevich distinguised square in which every morphism is stabiliser preserving.

\begin{proposition} \label{gysin}
	Let $\sX$ be a smooth Deligne-Mumford stack. Let $\sZ\subset \sX$ be a smooth closed substack of pure codimension $c$. Then there exists a Gysin triangle:
	\[\twistedM(\sX\setminus \sZ)\rightarrow \twistedM(\sX)\rightarrow\twistedM(\sZ)(c)[2c]\rightarrow\twistedM(\sX\setminus\sZ)[1]\]
\end{proposition}

\begin{proof}
	As $\sZ\rightarrow \sX$ is stabiliser perserving, we have a cartesian square
	\begin{center}
		\begin{tikzcd}
			C^t_\sZ\arrow[r]\arrow[d] & C^t_\sX\arrow[d]\\
			\sZ\arrow[r] &\sX.
		\end{tikzcd}
	\end{center}
	Moreover, note that $C^t_{\sX\setminus \sZ}=C^t_\sX \setminus C^t_\sZ$.
	Thus, we have a Gysin triangle of ordinary motives,
	\[M(C^t_\sX \setminus C^t_\sZ)\rightarrow M(C^t_\sX)\rightarrow M(\frac{C^t_\sX}{C^t_\sX \setminus C^t_\sZ})\rightarrow M(C^t_\sX \setminus C^t_\sZ)[1].\]
	Let $i:C^t_\sZ\rightarrow C^t_\sX$ denote the closed immersion induced on the character stacks, with the complementary open immersion $j: C^t_\sU\rightarrow C^t_\sX$. Here, $\sU:= \sX\setminus\sZ$ denotes the open complement of $\sZ$ in $\sX$.
	To prove the Gysin triangle, consider the following localisation sequence in $SH(\sX)$:
	\[j_! j^!\rightarrow Id \rightarrow i_* i^*\]
	As $j$ is an open immersion we have that $j^*\simeq j^!$, and, also, that $j_!\simeq j_\#$.
	Let $q: C^t_\sX\rightarrow\Spec k$ and $p: C^t_\sZ\rightarrow\Spec k$, denote the structure morphisms. By relative purity \cite[Proposition 2.4.31]{cd2019} and \cite[\S 4]{cdnonrepresentablesixfunctorformalisms}, , we have $q_\# i_*\simeq p_\# \Sigma^{N_\sZ \sX}$. Thus, applying $q_\#$ to the localisation sequence above gives,
	\[q_\# j_\# j^* \rightarrow Id \rightarrow p_\# \Sigma^{N_\sZ \sX} i^*\]
	Finally, observe that $\Lambda^\vee_\sZ$ and $\Lambda^\vee_\sU$ are the pullbacks of $\Lambda_\sX$ to the substacks $\sZ$ and $\sU$, respectively. Hence, applying the above localisation sequence to $\Lambda^\vee_\sX$ proves the result.
\end{proof}

\begin{proposition}
	Let $\sX$ be a Deligne-Mumford stack over a field $k$. Let $\sE$ be a vector bundle of rank $n+1$ such that $\sE\rightarrow \sX$ is stabiliser preserving. Then there exists an isomorphism in $\DM{k,\Q}{}$:
	\[\twistedM(\Proj(\sE))\rightarrow \overset{n}{\underset{i=0}{\bigoplus}} \twistedM (\sX)(i)[2i]\]
\end{proposition}

\begin{proof}
	We have a cartesian square
	\begin{center}
	\begin{tikzcd}
		C^t_{\P(\sE)}\arrow[d]\arrow[r] & \P(\sE)\arrow[d]\\
		C^t_\sX \arrow[r] & \sX
	\end{tikzcd}
	\end{center}
	As projective bundles are preserved under base change, we have $C^t_{\P(\sE)}\simeq \P(\sE_{C^t_\sX})$. This implies the result. 
\end{proof}

\subsection{Homotopy invariance}

\begin{lemma} \label{descent}
	Let $f : \mathcal{Y} \to \mathcal{X}$ be a morphism of 
	stacks. Let $x : X_{\bullet} \to \mathcal{X}$ and $y : Y_{\bullet} \to \mathcal{Y}$ and let $f' : Y_{\bullet} \to X_{\bullet}$ be morphisms of simplicial stacks such that 
	\begin{enumerate}
		\item $x$ and $y$ induces isomorphisms $$hocolim_{[n] \in \bigtriangleup}(X_n) \cong \mathcal{X}$$ and  $$hocolim_{[n] \in \bigtriangleup}(Y_n) \cong \mathcal{Y}$$ in $\mathbf{H}^{et}(k)$. 
		\item The following diagram commutes
		\begin{center}
			\begin{tikzcd}
			Y_{\bullet} \arrow{d}{f'} \arrow{r}{y} & \mathcal{Y} \arrow{d}{f}\\
			X_{\bullet} \arrow{r}{x} & \sX
			\end{tikzcd}
		\end{center}
		\item $f'_i$ induces isomorphism in $\DM{k, \Q}{}$ for each $i$.

	\end{enumerate}
	Then $f$ induces isomorphism in $\DM{k, \Q}{}$.
\end{lemma}
\begin{proof}
	This is precisely homotopy invariance of homotopy colimits. 
	
\end{proof}

\begin{theorem}
	Let $p : \mathcal{V} \to \mathcal{X}$ be a vector bundle over a smooth Deligne-Mumford stack $\mathcal{X}$,  then $\twistedM(\mathcal{V}) \cong \twistedM(\mathcal{X})$ in $\DM{k, \Q}{}$.
	
\end{theorem}
\begin{proof}
	By Galois descent and Theorem \ref{comp inertia} we are reduced to show that the map $p' : I_{\mathcal{V}} \to I_{\mathcal{X}}$ induced by $p$ induces an isomorphism $M( I_{\mathcal{V}}) \cong M( I_{\mathcal{X}})$ in  $\DM{k, \Q}{}$. Let $\coprod_i [U_i/H_i] \to \sX$ be an \'etale covering constructed from an etale covering of its coarse space. Let $X_{\bullet}$ be the corresponding Cech simplicial stack and $X_{\bullet}\to \sX$ is an etale weak equivalence. Note that each simplcial dimension stack $X_i$ of $X_{\bullet}$ is a quotient stack. Therefore $V_{\bullet} := X_{\bullet} \times_{\sX} \sV \to \sV$ is also an etale weak equivalence and $V_i \to X_i$ are vector bundles.  Let for any simplcial stack $Y_{\bullet}$, let $I_{Y_{\bullet}}$ be the simplicial stack such that  $(I_{Y_{\bullet}})_i := I_{Y_i}$. 
	
	\textbf{Claim } $I_{V_{\bullet}} \to I_{\sV}$ and $I_{X_{\bullet}} \to I_{\sX}$ are etale weak equivalence.\\
	\textbf{Proof of Claim}
	Let $f:\coprod_i [U_i/H_i] \to \sX$ be an \'etale covering. As $\sX$ has finite inertia, we can also arrange $f$ to be stabiliser preserving as the locus where $f$ is stabiliser preserving is open \cite[Proposition 3.5]{RydhQuotient}. This implies that the maps $I_{V_{0}} \to I_{\sV}$ and $I_{X_{0}} \to I_{\sX}$ admit  \'{e}tale local sections. This implies the claim.

	Therefore from the claim and lemma \ref{descent} we are reduced to the case where $\sX = [U/G]$ where $U$ is an affine scheme and $G$ is a finite group and $\sV$ is then given by a $G$-equivarent vector bundle over $U$. Therefore, there exists a $G$ equivariant vector bundle $V \to U$ such that $\sV = [V/G]$. In this case $I_{\sV} \to I_{\sX}$ is vector bundle.  
\end{proof}

\section{Computations}\label{section-computations}
In this section, we present some computations of the motive $M_\chi (F)$ for various Deligne-Mumford stacks.

\begin{example}[The twisted motive of $\overline{\mathcal{M}}_{1,n}$]
Using the computations of \cite{pagani-chen-ruan}, we can also the motive $M_\chi (\overline{\mathcal{M}}_{1,n})_{\Q(\mu_\infty)}$, which by Theorem \ref{comp inertia} is the motive of the inertia stack $I(\overline{\mathcal{M}}_{1,n})$ with $\Q(\mu_\infty)$-coefficients. We refer the reader to \textit{loc. cit.} for notations.

In this case, each component $F_k$ of $I(\overline{\mathcal{M}}_{1,n})$ has the form \cite[Corollary 3.23]{pagani-chen-ruan},
\[F_k = A\times \overline{\mathcal{M}}_{0,n_1}\times \overline{\mathcal{M}}_{0,n_2} \times \overline{\mathcal{M}}_{0,n_3}\times \overline{\mathcal{M}}_{0,n_4}\]
where $n_i\geq 3$ and $A\in \lbrace B\mu_3, B\mu_4, B\mu_6, \P(4,6),\P(2,4),\P(2,2)\rbrace$. Thus, the motive $M_\chi (\overline{\mathcal{M}}_{1,n})_{\Q(\mu_\infty)}$ is a disjoint union of the motives of $M(F_i)_{\Q(\mu_\infty)}$.

For $\overline{\mathcal{M}}_{1,1}$, this turns out to be \cite[Corollary]{pagani-chen-ruan},
\begin{align*}
M_\chi (\overline{\mathcal{M}}_{1,1})_{\Q(\mu_\infty)} &= (M(\overline{\mathcal{M}}_{1,1})^{\oplus 2}\oplus M(B\mu_4)^{\oplus 2}\oplus M(B\mu_6)^{\oplus 4})_{\Q(\mu_\infty)}\\
&=\Q(\mu_\infty)^{\oplus 8}\oplus (\Q(\mu_\infty)(1)[2])^{\oplus 2}.
\end{align*}

\end{example}

\begin{example}
Let $\sX$ be a smooth Deligne-Mumford stack.Consider the decomposition of $C^t_{\sX} = \sX \amalg C^t_{\sX, +}$ coming from the section $\sX \to C^t_{\sX} $ given by $(u) \mapsto (u, \left\{e \right\})$. The sheaf $\Lambda_{\sX}$ restricted to the component $\sX$ is $\mathbb{Q}$. Therefore $M(\sX)$ is a direct factor of $\twistedM(\sX)$. Following Toen ( \cite{Toen}) we define the complimentary direct factor of $\twistedM(\sX)$ by $M_{\chi \neq 1}(\sX)$.  Then $\twistedM(\sX) = M(\sX) \oplus M_{\chi \neq 1}(\sX)$. The part  $M_{\chi \neq 1}(\sX)$ can be simplified in the following cases.

Let $X$ be a smooth $k$ acheme and let $G$ be a finite group acting on it. Let $c(G)$ be a set of representatives of conjugacy classes of cyclic subgroups whose order is relatively prime to $char(k)$. For $c \in c(G)$ let $X^c$ be the fixed locus of of $c$,  $s(c)$ denote the set of injective characters of $c$ and let $N_c$ be the normaliser of $c$.  Note that
$C^t_{[X/G]} \cong \amalg_{c \in c(G)}[X^c/N_c]$. For $\sX = [X/G]$, it is easy to see that $\Lambda_{\sX}^{\vee}|_{X^c} \cong \oplus_{s(c)}1_{X^c}$. The group $N_c$ acts via conujgation on $X^c \times s(c)$ and let $j : [X^c \times s(c)/ N_c] \to [X^c/N_c]$ be the corresponding map of stacks. Then $\Lambda_{\sX}^{\vee}|_{ [X^c/N_c]} \cong j_{\sharp}(1)$. Let $f : C^t_{[X/G]} \to Spec(k)$ be the structure map. Then $$f_{\sharp}(\Lambda_{\sX}^{\vee}) \cong \oplus_{c \in c(G)} f_{\sharp}(\Lambda_{\sX}^{\vee}|_{ [X^c/N_c]}) \cong $$
$$\cong \oplus_{c \in c(G)} M([X^c \times s(c)/ N_c]) \cong  \oplus_{c \in c(G)} M(X^c \times s(c))^{N_c}.$$

\end{example}

\begin{example}

Let $\sX$ be a Gerbe, therefore $C^t_{\sX} \to \sX$ is finite \'etale and locally for the \'etale topology on the coarse moduli space $\sX \to X$, the stack $\sX$ is isomorphic to $X \times BH$ for some finite group $H$. Following Toen (\cite[\S  Examples 2]{Toen}) , there exists a finite \'etale covering $Y \to X$, suhc that $\twistedM(\sX) \cong M(Y)$. Moreover $Y \to X$ has a section and therefore $M(Y) \cong M(X) \oplus M'$, where $M(X) \cong M(\sX)$.
\end{example}
\

\begin{theorem}
Let $\sX$ be a smooth proper Deligne-Mumford stack with a $\G_m$ action. Let $x : X \to \sX$ be a $\G_m$-equivariant affine \'etale surjective atlas. Let $\sZ := \sX^{\G_m} = \coprod_i \sZ_i$ be the decomposition into connected components of the fixed point locus. Then  $\twistedM(\sX) \cong \oplus_{i=0}^n \twistedM(\sZ_i)(c_i)[2c_i]$, for $c_i$'s in $\mathbb{N}$. 
\end{theorem}
\begin{proof}
Let $\sX$ be a smooth proper Deligne-Mumford stack with a $\G_m$ action and let $x : X \to \sX$ be a $\G_m$-equivariant affine \'{e}tale atlas. Let $\sZ$ be the fixed point closed substack $\sX$ and let $\sZ_i$ be connected componenets of $\sZ$. Then by \cite[Proposition 5]{oprea-tautological-classes}, we get that $\sX$ can be covered by locally closed disjoint substacks $\sZ_i^+$, containing $\sZ_i$  which are affine fibrations over $\sZ_i$. Note that $C^t_{\sZ_i^+} = \sZ_i^+ \times_{\sX}  C^t_{\sX}$ are locally closed  disjoint substacks covering  $C^t_{\sX}$ and $C^t_{\sZ_i} = \sZ_i \times_{\sX}  C^t_{\sX} = \sZ_i \times_{\sZ_i^+}  C^t_{\sZ_i^+}$. Therefore, the affine fibration $\sZ_i^+ \to \sZ_i$,  induces an affine fibration  $\sZ_i^+ \times_{\sZ_i} C^t_{\sZ_i} \to  C^t_{\sZ_i} $, but $\sZ_i^+ \times_{\sZ_i} C^t_{\sZ_i}\cong  C^t_{\sZ_i^+}$, by what we have described. Therefore we get a covering of $C^t_{\sX}$ by disjoint localy closed $C^t_{\sZ_i^+}$ containing  $C^t_{\sZ_i}$ together with affine fibrations $C^t_{\sZ_i^+} \to C^t_{\sZ_i}$. Therefore, $M(C^t_{\sZ_i^+} ) \cong  M(C^t_{\sZ_i})$ (using utsav choudhury and skowera, corollary 3.2), which implies that $\twistedM(\sZ_i^+) \cong \twistedM(\sZ_i)$. As $\sX$ is smooth proper so are $\sZ_i$'s. Therefore $\twistedM(\sZ_i^+)$ are Chow motives by Theorem \ref{geometric}. Therefore using similar arguments as in \cite[Proposition 4.4]{choudhury-skowera-cellular} together with Proposition \ref{gysin}, we get $\twistedM(\sX) \cong \oplus_{i=0}^n \twistedM(\sZ_i^+)(c_i)[2c_i]$, where $c_i =codim_{\sX}(\sZ_i^+)$. 
\end{proof}

\begin{theorem}
Let $\sX$ be a smooth proper Deligne-Mumford stack with a $\G_m$ action. Let $x : X \to \sX$ be a $\G_m$-equivariant affine \'etale surjective atlas. Let $C^t_{\sX}$ be endowed with the $\G_m$ action which makes the morphism $C^t_{\sX} \to \sX$ a $\G_m$-invariant morphism. Let $\sZ' :=  C^t_{\sX} = \coprod_j \sZ'_j$ be the decomposition into fixed point locus. Then
 $\twistedM(\sX) \cong \oplus_{i=0}^n M(\sZ'_i)_{\chi}(d_i)[2d_i]$, for $d_i$'s in $\mathbb{N}$ and  for any smooth substack $\sY \subset C^t_{\sX}$ we define  $M(\sY)_{\chi}:= q_{\#}(\Lambda^\vee|_{\sY})$ where $q : \sY \to Spec(k)$ is the structure maps.

\end{theorem}

\begin{proof}
Suppose $\sX$ is a smooth  Deligne-Mumford stack such with $\G_m$ action such that there exists a $\G_m$ equivariant affine \'etale surjective atlas $x : X \to \sX$. Then the \'etale affine atlas $\pi^*(x) : X \times_{\sX} C^t_{\sX} \to  C^t_{\sX}$ is also $\G_m$ equivariant. Therefore again by \cite[Proposition 5]{oprea-tautological-classes}, we get that $C^t_{\sX}$ can be covered by locally closed disjoint substacks $\sZ_i^{+,'}$, containing $\sZ'_i$  which are affine fibrations over $\sZ'_i$ . Note that for any smooth closed substack $\sY \subset C^t_{\sX}$ of codimension $c$ and open compliment $\sU$ using the same argument as Proposition \ref{gysin} we get the Gysin triangle:
	\[M(\sU)_{\chi} \rightarrow \twistedM(\sX)\rightarrow M(\sY)_{\chi}(c)[2c]\rightarrow M(\sU)_{\chi}[1]\]

As $M(\sZ_i^{+,'}) \cong M(\sZ'_i)$, using \cite[Corollary 3.2]{choudhury-skowera-cellular}), this implies $M(\sZ_i^{+,'})_{\chi} \cong M(\sZ'_i)_{\chi}$. The same method of Theorem \ref{geometric} shows that $M(\sZ'_i)_{\chi}$'s are Chow motives. Now our claim follows exactly the same way as in \cite[Proposition 4.4]{choudhury-skowera-cellular}. Here $d_i$ is the codimension of $\sZ_i^{+,'}$.
\end{proof}

\bibliography{mybib.bib}
\bibliographystyle{alphanum}

\end{document}